\documentclass[12pt]{amsart}

\usepackage{amsmath, amssymb, amscd, newlfont, bbm}
\usepackage[foot]{amsaddr}
\usepackage[final]{hyperref}
\hypersetup{
	colorlinks=true,        
	linkcolor={black},  
	citecolor={black},
	urlcolor={black}     
}

\newcommand{\Met}{{\mathrm{SL}_2(\mathbb R)^{\sim}}}

\newcommand{\spacedcdot}{{\,\cdot\,}}

\newcommand{\bbC}{{\mathbb{C}}}

\newcommand{\bbR}{{\mathbb{R}}}
\newcommand{\bbZ}{{\mathbb{Z}}}

\newcommand{\bfGamma}{{\mathbf{\Gamma}}}

\newcommand{\Lie}{{\mathrm{Lie}}}

\newcommand{\M}{{\mathrm{M}}}

\DeclareMathOperator{\Hol}{{\mathrm{Hol}}}

\newcommand{\calA}{{\mathcal{A}}}
\newcommand{\calC}{{\mathcal{C}}}

\newcommand{\calH}{{\mathcal{H}}}

\newcommand{\frakg}{{\mathfrak{g}}}

\newcommand{\fraksl}{{\mathfrak{sl}}}

\usepackage[mathscr]{eucal}

\def \ScptH{\mathcal H}

\def\SL#1{{\mathrm{SL}}_{#1}}

\providecommand{\abs}[1]{\left\lvert#1\right\rvert}
\providecommand{\norm}[1]{\left\lVert#1\right\rVert}
\providecommand{\scal}[2]{\left<#1,#2\right>}

\numberwithin{equation}{section}

\newtheorem{Prop}[equation]{Proposition}
\newtheorem{Lem}[equation]{Lemma}

\newtheorem{Thm}[equation] {Theorem}
\newtheorem{Cor}[equation]{Corollary}

\title
   [ Poincar\' e series of half-integral weight]
   { On Poincar\' e series of half-integral weight}
\author{Sonja \v Zunar}

\address{ Department of Mathematics,
Faculty of Science,
University of Zagreb,
Bijeni\v cka 30,
10000 Zagreb,
Croatia}
\email{szunar@math.hr}

\thanks{The author acknowledges Croatian Science Foundation grant no. 9364.}

\begin{document}
\maketitle

\begin{abstract}
	We use Poincar\' e series of $ K $-finite matrix coefficients of genuine integrable representations of the metaplectic cover of $ \SL2(\bbR) $ to construct a spanning set for the space of cusp forms $ S_m(\Gamma,\chi) $, where $ \Gamma $ is a discrete subgroup of finite covolume in the metaplectic cover of $ \SL2(\bbR) $, $ \chi $ is a character of $ \Gamma $ of finite order, and $ m\in\frac52+\bbZ_{\geq0} $. We give a result on the non-vanishing of the constructed cusp forms and compute their Petersson inner product with any $ f\in S_m(\Gamma,\chi) $. Using this last result, we construct a Poincar\' e series $ \Delta_{\Gamma,k,m,\xi,\chi}\in S_m(\Gamma,\chi) $ that corresponds, in the sense of the Riesz representation theorem, to the linear functional $ f\mapsto f^{(k)}(\xi) $ on $ S_m(\Gamma,\chi) $, where $ \xi\in\bbC_{\Im(z)>0} $ and $ k\in\bbZ_{\geq0} $. Under some additional conditions on $ \Gamma $ and $ \chi $, we provide the Fourier expansion of cusp forms $ \Delta_{\Gamma,k,m,\xi,\chi} $ and their expansion in a series of classical Poincar\' e series.
\end{abstract}

\section{Introduction}

In this paper, we adapt representation-theoretic techniques developed for the group $ \SL2(\bbR) $ in \cite{MuicJNT} and \cite{MuicInner} to the case of the metaplectic cover of $ \SL2(\bbR) $. Using this, we prove a few results on cusp forms of half-integral weight. 

To give an overview of our results, we introduce the basic notation. The metaplectic cover of $ \SL2(\bbR) $ can be realized as the group
\[ \Met:=\left\{\sigma=\left(g_\sigma=\begin{pmatrix}a&b\\c&d\end{pmatrix},\eta_\sigma\right)\in\SL2(\bbR)\times\Hol(\calH):\eta_\sigma^2(z)=cz+d\text{ for all }z\in\calH\right\}, \]
where $ \Hol(\calH) $ is the space of all holomorphic functions defined on the upper half-plane $ \calH $. The multiplication law and smooth structure of $ \Met $ are defined in Section \ref{sec:069}.  $ \Met $ acts on $ \calH $ by $ \left(\begin{pmatrix}a&b\\ c&d\end{pmatrix},\eta\right).z:=\frac{az+b}{cz+d} $. Moreover, for every $ m\in\frac12+\bbZ_{\geq0} $ we have the following right action of $ \Met $ on $ \bbC^{\calH} $: $ \left(f\big|_m\sigma\right)(z):=f(\sigma.z)\eta_\sigma(z)^{-2m}. $ Let $ P:\Met\to\SL2(\bbR) $ be the projection onto the first coordinate.

Next, let $ \Gamma $ be a discrete subgroup of finite covolume in $ \Met $, $ \chi:\Gamma\to\bbC^\times $ a character of finite order, and $ m\in\frac32+\bbZ_{\geq0} $. The space $ S_m(\Gamma,\chi) $ of cusp forms of weight $ m $ for $ \Gamma $ with character $ \chi $ by definition consists of all $ f\in\Hol(\calH) $ that satisfy $ f\big|_m\gamma=\chi(\gamma)f $ for all $ \gamma\in\Gamma $ and vanish at all cusps of $ P(\Gamma) $. $ S_m(\Gamma,\chi) $ is a finite-dimensional Hilbert space under the Petersson inner product $ \scal{f_1}{f_2}_{\Gamma}:=\abs{\Gamma\cap Z\left(\Met\right)}^{-1}\int_{\Gamma\backslash\calH}f_1(z)\overline{f_2(z)}\Im(z)^m\,dv(z) $, where $ dv(x+iy):=\frac{dx\,dy}{y^2} $. We write $ S_m(\Gamma):=S_m(\Gamma,1) $. Let us denote by $ \Phi $ the classical lift $ S_m(\Gamma)\to\bbC^\Met $ that maps $ f\in S_m(\Gamma) $ to $ F_f:\Met\to\bbC $, $ F_f(\sigma):=\left(f\big|_m\sigma\right)(i) $, where $ i $ is the imaginary unit. $ \Phi $ is a
unitary isomorphism $ S_m(\Gamma)\to\Phi(S_m(\Gamma))=:\calA\left(\Gamma\backslash\Met\right)_m\subseteq L^2\left(\Gamma\backslash\Met\right) $ (Theorem \ref{thm:005}).

The starting point of this paper are results of \cite{ZunarNonVan2017}, where we applied the techniques of \cite{MuicJNT} to compute certain $ K $-finite matrix coefficients of genuine integrable representations of $ \Met $ and study their Poincar\' e series with respect to $ \Gamma $. In Lemma \hyperref[lem:045:5]{\ref*{lem:045}.(\ref*{lem:045:5})}, we show that the Poincar\' e series $ P_\Gamma F_{k,m} $ ($ k\in\bbZ_{\geq0} $, $ m\in\frac52+\bbZ_{\geq0} $) discussed in \cite[Section 6]{ZunarNonVan2017} belong to $ \calA\left(\Gamma\backslash\Met\right)_m $. The main result of this paper is Theorem \ref{thm:017}, in which we compute the Petersson inner product of $ P_\Gamma F_{k,m} $ with any $ \varphi\in\calA\left(\Gamma\backslash\Met\right)_m $ using the representation theory of $ \Met $. It is the $ \Met $-variant of \cite[Theorem 2-11]{MuicInner}.

In the rest of the paper, we use the facts of the previous paragraph to prove a few results about $ S_m(\Gamma,\chi) $ for $ m\in\frac52+\bbZ_{\geq0} $. Most of these results are half-integral weight variants of results of \cite{MuicJNT}, \cite{MuicCurves}, \cite{MuicLFunk}, and \cite{MuicInner}. 

First,
by considering the preimages of functions $ P_\Gamma F_{k,m} $ under $ \Phi $, in Theorem \ref{thm:020} we construct the following spanning set for $ S_m(\Gamma,\chi) $ (cf.\ \cite[Lemma 4-2]{MuicJNT}):
\begin{equation}\label{eq:073}
\left(P_{\Gamma,\chi}f_{k,m}\right)(z):=(2i)^m\sum_{\gamma\in\Gamma}\overline{\chi(\gamma)}\frac{\left(\gamma.z-i\right)^k}{\left(\gamma.z+i\right)^{m+k}}\eta_\gamma(z)^{-2m},\qquad z\in\calH,\ k\in\bbZ_{\geq0},
\end{equation}
(see \eqref{eq:091}). Moreover, we obtain results (Theorem \ref{thm:030} and Corollary \ref{cor:070}) on the non-vanishing of cusp forms $ P_{\Gamma,\chi}f_{k,m} $ in the case when $ P(\Gamma)\subseteq\SL2(\bbZ) $ by adapting our study of the non-vanishing of functions $ P_\Gamma F_{k,m} $ conducted in \cite[Section 6]{ZunarNonVan2017}. 

Next, Theorem \ref{thm:017} translates via the unitary isomorphism $ \Phi^{-1} $ to Theorem \hyperref[thm:020:3]{\ref*{thm:020}.(\ref*{thm:020:3})}, which states that, for every $ k\in\bbZ_{\geq0} $,
\[ \scal f{P_{\Gamma,\chi}f_{k,m}}_\Gamma=\sum_{l=0}^k\binom kl(2i)^l\frac{4\pi}{\prod_{r=0}^{l}(m-1+r)}f^{(l)}(i),\qquad f\in S_m(\Gamma,\chi). \] 
It is a short way from this relation to the proof of the following fact in the case when $ \xi=i $: for every $ k\in\bbZ_{\geq0} $, the Poincar\' e series
\begin{equation}\label{eq:074}
\Delta_{\Gamma,k,m,\xi,\chi}(z):= \frac{(2i)^m}{4\pi}\left(\prod_{r=0}^{k}(m-1+r)\right)\sum_{\gamma\in\Gamma}\frac{\overline{\chi(\gamma)}}{\left(\gamma.z-\overline\xi\right)^{m+k}}\eta_\gamma(z)^{-2m},\qquad z\in\calH, 
\end{equation}
belongs to $ S_m(\Gamma,\chi) $ and satisfies
\[ \scal f{\Delta_{\Gamma,k,m,\xi,\chi}}_\Gamma=f^{(k)}(\xi),\qquad f\in S_m(\Gamma,\chi).  \]
We prove that this holds for all $ \xi\in\calH $ in Proposition \ref{prop:026} and Theorem \ref{thm:048} (cf.\ \cite[Corollary 1-2]{MuicInner}). 

Incidentally, our proof of Theorem \ref{thm:048} proves the following integral formula (Corollary \ref{cor:071}): 
\begin{equation}\label{eq:083}
f^{(k)}(\xi)=\frac{(-2i)^m}{4\pi}\left(\prod_{r=0}^k(m-1+r)\right)\int_\calH\frac{f(z)}{\left(\overline z-\xi\right)^{m+k}}\Im(z)^m\,dv(z)
\end{equation}
for all $ f\in S_m(\Gamma,\chi) $, $ k\in\bbZ_{\geq0}$, and $ \xi\in\calH $. We use this formula in Corollary \ref{cor:081} to give a short proof that
\begin{equation}\label{eq:089}
\sup_{\xi\in\calH}\abs{f^{(k)}(\xi)\Im(\xi)^{\frac m2+k}}<\infty,\qquad f\in S_m(\Gamma,\chi),\ k\in\bbZ_{\geq0},
\end{equation}
which enables us to prove, in Proposition \ref{prop:082}, that
\[ \sup_{z,\xi\in\calH}\Im(\xi)^{\frac m2+k}\Im(z)^{\frac m2}\abs{\Delta_{\Gamma,k,m,\xi,\chi}(z)}<\infty  \]
for every $ k\in\bbZ_{\geq0} $ (cf.\ \cite[(1-5)]{MuicLFunk}).

Next, assume that $ \infty $ is a cusp of $ P(\Gamma) $ and that $ \eta_\gamma^{-2m}=\chi(\gamma) $ for all $ \gamma\in\Gamma_\infty $, so that we have the classical Poincar\' e series $ \psi_{\Gamma,n,m,\chi}\in S_m(\Gamma,\chi) $, $ n\in\bbZ_{>0} $, defined by
\[ \psi_{\Gamma,n,m,\chi}(z):=\sum_{\gamma\in\Gamma_\infty\backslash\Gamma}\overline{\chi(\gamma)}e^{2\pi in\frac{\gamma.z}h}\eta_\gamma(z)^{-2m},\qquad z\in\calH, \]
where $ h\in\bbR_{>0} $ is such that the group $ \left\{\pm 1\right\}P(\Gamma_\infty) $ is generated by $ \left\{\pm\begin{pmatrix}1&h\\&1\end{pmatrix}\right\} $. 
Theorem \ref{thm:035} gives the Fourier expansion of cusp forms $ \Delta_{\Gamma,k,m,\xi,\chi} $ and their expansion in a series of classical Poincar\' e series (cf.\ \cite[Theorem 3-5]{MuicCurves}). In Corollary \ref{cor:090}, this Fourier expansion combined with \eqref{eq:089} provides a quick proof of some bounds on the derivatives of classical Poincar\' e series (cf.\ \cite[Theorem 1-2]{MuicLFunk}).

Finally, in Section \ref{sec:084} we apply our results to the standardly defined spaces $ S_m(N,\chi) $, where $ N\in4\bbZ_{>0} $ and $ \chi $ is an even Dirichlet character modulo $ N $ (e.g., see \cite{shimura}). We show that $ S_m(N,\chi) $ coincides with $ S_m\left(\bfGamma_0(N),\chi\right) $, where $ \bfGamma_0(N) $ is an appropriate discrete subgroup of $ \Met $, and $ \chi $ is identified with a suitable character of $ \bfGamma_0(N) $. Corollary \ref{cor:066} gives a formula for the action of Hecke operators $ T_{p^2,m,\chi} $, for prime numbers $ p\nmid N $, on cusp forms $ \Delta_{\bfGamma_0(N),k,m,\xi,\chi} $ in terms of their expansion in a series of classical Poincar\' e series (cf.\ \cite[Lemma 5-8]{MuicCurves}).

Let us mention that a non-re\-pre\-sen\-ta\-tion-the\-o\-re\-tic proof of Proposition \ref{prop:026} and Theorem \ref{thm:048} in the case when $ k=0 $ can be obtained by adapting the proof of \cite[Theorem 6.3.3]{miyake} to half-integral weights. The case when $ k\in\bbZ_{>0} $ can be derived from it essentially by taking the $ k $th derivative (the details can be gleaned from the first sentence of the proof of Proposition \ref{prop:026} and from Lemma \ref{lem:084}). Similarly, the integral formula \eqref{eq:083} can be deduced from the half-integral weight variant of \cite[Theorem 6.2.2]{miyake}; the integral-weight variant of \eqref{eq:083} for $ k=0 $ is actually used in the proof of \cite[Theorem 6.3.3]{miyake} (see the last equality on \cite[pg.~230]{miyake}).   
On the other hand, our results on the non-vanishing of cusp forms $ P_{\Gamma,\chi} f_{k,m} $ are based on applying the integral criterion \cite[Lemma 2-1]{MuicIJNT} to the corresponding Poincar\' e series on $ \Met $. To do that, we used the Cartan decomposition of $ \Met $, which is not easily accessible when working directly in $ S_m(\Gamma,\chi) $.  

\vspace{3mm}
This paper grew out of my PhD thesis. I would like to thank my advisor, Goran Mui\' c, for his support, encouragement, and many discussons. I am also grateful to Neven Grbac and Marcela Hanzer for their continued support and useful comments.

\section{Preliminaries on the metaplectic group}\label{sec:069}

Let $ \sqrt\spacedcdot:\bbC\to\bbC $ be the branch of the complex square root with values in $ \left\{z\in\bbC:\Re(z)>0\right\}\cup\left\{z\in\bbC:\Re(z)=0,\Im(z)\geq0\right\} $. We write $ i:=\sqrt{-1} $ and 
\begin{equation}\label{eq:091}
z^m:=\left(\sqrt z\right)^{2m},\qquad z\in\bbC^\times,\ m\in\frac12+\bbZ. 
\end{equation}
Next, we define $ \calH:=\left\{z\in\bbC:\Im(z)>0\right\} $ and denote by $ \Hol(\calH) $ the space of all holomorphic functions $ \calH\to\bbC $. 

The group $ \SL2(\bbR) $ acts on $ \bbC\cup\{\infty\} $ by
\[ g.z:=\frac{az+b}{cz+d},\qquad g=\begin{pmatrix}a&b\\c&d\end{pmatrix}\in\SL2(\bbR),\ z\in\bbC\cup\{\infty\}. \]
We have
\begin{equation}\label{eq:056}
	\Im(g.z)=\frac{\Im(z)}{\abs{cz+d}^2},\qquad g=\begin{pmatrix}a&b\\c&d\end{pmatrix}\in\SL2(\bbR),\ z\in\calH. 
\end{equation}
For every $ N\in\bbZ_{>0} $, we denote
\begin{align*}
	\Gamma_0(N)&:=\left\{\begin{pmatrix}a&b\\c&d\end{pmatrix}\in\SL2(\bbZ):c\equiv0\pmod N\right\},\\
	\Gamma_1(N)&:=\left\{\begin{pmatrix}a&b\\c&d\end{pmatrix}\in\SL2(\bbZ):c\equiv0,\ a\equiv d\equiv1\pmod N\right\},\\
	\Gamma(N)&:=\left\{\begin{pmatrix}a&b\\c&d\end{pmatrix}\in\SL2(\bbZ):b\equiv c\equiv0,\ a\equiv d\equiv1\pmod N\right\}.
\end{align*}

The group
\[ \Met:=\left\{\sigma=\left(g_\sigma=\begin{pmatrix}a&b\\c&d\end{pmatrix},\eta_\sigma\right)\in\SL2(\bbR)\times\Hol(\calH):\eta_\sigma^2(z)=cz+d\text{ for all }z\in\calH\right\} \]
with multiplication law
\begin{equation}\label{eq:075}
\sigma_1\sigma_2:=\left(g_{\sigma_1}g_{\sigma_2},\eta_{\sigma_1}(g_{\sigma_2}.z)\eta_{\sigma_2}(z)\right),\qquad\sigma_1,\sigma_2\in\Met,
\end{equation}
acts on $ \bbC\cup\{\infty\} $ by
\[ \sigma.z:=g_\sigma.z,\qquad \sigma\in\Met,\ z\in\bbC\cup\{\infty\}, \]
and, for every $ m\in\frac12+\bbZ_{\geq0} $, on the right on $ \bbC^\calH $ by
\[ \left(f\big|_m\sigma\right)(z):=f(\sigma.z)\eta_\sigma(z)^{-2m},\qquad z\in\calH,\ f\in\bbC^\calH,\ \sigma\in\Met. \]

In the following, we use shorthand notation $ (g_\sigma,\eta_\sigma(i)) $ for elements $ \sigma=(g_\sigma,\eta_\sigma) $ of $ \Met $. $ \Met $ is a connected Lie group with a smooth (Iwasawa) parametrization $ \bbR\times\bbR_{>0}\times\bbR\to\Met $,
\begin{equation}\label{eq:001}
(x,y,t)\mapsto\left(\begin{pmatrix}1&x\\0&1\end{pmatrix},1\right)\left(\begin{pmatrix}y^{\frac12}&0\\0&y^{-\frac12}\end{pmatrix},y^{-\frac14}\right)\left(\begin{pmatrix}\cos t&-\sin t\\\sin t&\cos t\end{pmatrix},e^{i\frac t2}\right).
\end{equation} 
The projection $ P:\Met\to\SL2(\bbR) $ onto the first coordinate is a smooth covering homomorphism of degree $ 2 $. The center of $ \Met $ is $ Z\left(\Met\right):=P^{-1}\left(\left\{\pm 1\right\}\right)\cong(\bbZ/4\bbZ,+) $.

We will denote the three factors on the right-hand side of \eqref{eq:001}, from left to right, by $ n_x $, $ a_y $, and $ \kappa_t $. We also define
\[ h_t:=\left(\begin{pmatrix}e^t&0\\0&e^{-t}\end{pmatrix},e^{-\frac t2}\right),\qquad t\in\bbR_{\geq0}. \]
Next, we recall the $ \SL2(\bbR) $-invariant Radon measure $ v $ on $ \calH $ defined by $ dv(x+iy):=\frac{dx\,dy}{y^2} $, $ x\in\bbR$, $ y\in\bbR_{>0} $, and fix the following Haar measure on $ \Met $: for $ \varphi\in C_c\left(\Met\right) $,
\begin{align}
\int_{\Met}\varphi\,d\mu_{\Met}:=&\frac1{4\pi}\int_0^{4\pi}\int_\ScptH \varphi\left(n_xa_y\kappa_t\right)\,dv(x+iy)\,dt\label{eq:035}\\ 
\,=&\frac1{4\pi}\int_0^{4\pi}\int_0^\infty\int_0^{4\pi}\varphi\left(\kappa_{\theta_1}h_t\kappa_{\theta_2}\right)\sinh(2t)\,d\theta_1\,dt\,d\theta_2.\label{eq:008}
\end{align}
Furthermore, for a discrete subgroup $ \Gamma $ of $ \Met $, let $ \mu_{\Gamma\backslash\Met} $ be the unique Radon measure on $ \Gamma\backslash\Met $ such that, for all $  \varphi\in C_c\left(\Met\right) $,
\[ \int_{\Gamma\backslash\Met}\sum_{\gamma\in\Gamma}\varphi(\gamma\sigma)\,d\mu_{\Gamma\backslash\Met}(\sigma)=\int_\Met \varphi\,d\mu_\Met. \]
Equivalently, for all $ \varphi\in C_c\left(\Gamma\backslash\Met\right) $,
\begin{equation}\label{eq:046}
\int_{\Gamma\backslash\Met}\varphi\,d\mu_{\Gamma\backslash\Met}=\frac1{4\pi\varepsilon_\Gamma}\int_0^{4\pi}\int_{\Gamma\backslash\calH} \varphi\left(n_xa_y\kappa_t\right)\,dv(x+iy)\,dt,
\end{equation}
where $ \varepsilon_\Gamma:=\abs{\Gamma\cap Z\left(\Met\right)} $. For every $ p\in\bbR_{\geq1} $, we define $ L^p\left(\Met\right) $ and $ L^p\left(\Gamma\backslash\Met\right) $ using $ \mu_\Met $ and $ \mu_{\Gamma\backslash\Met} $,  respectively. 

We identify the Lie algebra $ \frakg:=\Lie\left(\Met\right) $ with $ \Lie\left(\SL2(\bbR)\right)\equiv\fraksl_2(\bbR) $ via the differential of $ P $ at $ 1 $ and extend this identification to that of the universal enveloping algebras of their complexifications: $ U\left(\frakg_\bbC\right)\equiv U\left(\fraksl_2(\bbC)\right) $. Now,
\[ k^\circ:=\begin{pmatrix}0&-i\\i&0\end{pmatrix},\quad n^+:=\frac12\begin{pmatrix}1&i\\i&-1\end{pmatrix},\quad n^-:=\frac12\begin{pmatrix}1&-i\\-i&-1\end{pmatrix} \]
form a standard basis of $ \frakg_\bbC $ (we have $ \left[k^\circ,n^+\right]=2n^+ $, $ \left[k^\circ,n^-\right]=-2n^- $, and $ \left[n^+,n^-\right]=k^\circ $), and
\[ \calC:=\frac12\left(k^\circ\right)^2+n^+n^-+n^-n^+ \]
generates the center of $ U(\frakg_\bbC) $. We will need the formulae \cite[(2-13)--(2-14)]{ZunarNonVan2017} giving the action of $ \calC $ and $ n^+ $ as left-invariant differential operators on $ C^\infty\left(\Met\right) $ in Iwasawa coordinates:
\begin{align}
\mathcal C&=2y^2\left(\frac{\partial^2}{\partial x^2}+\frac{\partial^2}{\partial y^2}\right)+2y\frac{\partial^2}{\partial x\,\partial t},\label{eq:036}\\
n^+&=iye^{-2it}\left(\frac\partial{\partial x}-i\frac\partial{\partial y}\right)+\frac i2e^{-2it}\frac{\partial}{\partial t}.\label{eq:011}
\end{align}
$ n^- $ acts as the complex conjugate of $ n^+ $:
\begin{equation}\label{eq:037}
n^-=-iye^{2it}\left(\frac\partial{\partial x}+i\frac\partial{\partial y}\right)-\frac i2e^{2it}\frac{\partial}{\partial t}.
\end{equation}

$ K:=\{\kappa_t:t\in\bbR\} $ is a maximal compact subgroup of $ \Met $. It is isomorphic to $ (\bbR/4\pi\bbZ,+) $ via $ \kappa_t\mapsto t+4\pi\bbZ $. Its unitary dual consists of the characters $ \chi_n $, $ n\in\frac12\bbZ $, defined by
\[ \chi_n(\kappa_t):=e^{-int},\qquad t\in\bbR. \]
We say that a function $ F:\Met\to\bbC $ transforms on the left (resp., on the right) as $ \chi_n $ if $ F(\kappa\sigma)=\chi_n(\kappa)F(\sigma) $ (resp., $ F(\sigma\kappa)=F(\sigma)\chi_n(\kappa) $) for all $ \kappa\in K $ and $ \sigma\in\Met $. 

\section{Preliminaries on cusp forms of half-integral weight}

Let $ m\in\frac32+\bbZ_{\geq0} $. Let $ \Gamma $ be a discrete subgroup of finite covolume in $ \Met $. We denote by $ S_m(\Gamma) $ the space of cusp forms for $ \Gamma $ of weight $ m $, i.e., the space of all $ f\in\Hol(\calH) $ that satisfy
\begin{equation}\label{eq:003}
f\big|_m\gamma=f,\qquad\gamma\in\Gamma,
\end{equation}
and vanish at every cusp of $ P(\Gamma) $. Let us explain the last condition. For a cusp $ x $ of $ P(\Gamma) $, let $ \sigma\in\Met $ such that $ \sigma.\infty=x $. Then, it follows from \cite[Theorem 1.5.4.(2)]{miyake} that
\[ Z\left(\Met\right)\sigma^{-1}\Gamma_x\sigma=Z\left(\Met\right)\left<n_h\right> \]
for some $ h\in\bbR_{>0} $, hence $ f\big|_m\sigma $ has a Fourier expansion of the form
\[ \left(f\big|_m\sigma\right)(z)=\sum_{n\in\bbZ}a_ne^{\pi in\frac z{2h}},\qquad z\in\calH. \]
We say that $ f $ vanishes at $ x $ if $ a_n=0 $ for all $ n\in\bbZ_{\leq0} $. 

Next, we recall the half-integral weight variant of \cite[Theorems 2.1.5 and 6.3.1]{miyake}:

\begin{Lem}\label{lem:038}
	Let $ f\in\Hol(\calH) $ such that \eqref{eq:003} holds. Then, the following claims are equivalent:
	\begin{enumerate}
		\item $ f\in S_m(\Gamma) $.
		\item $ \sup_{z\in\calH}\abs{f(z)\Im(z)^{\frac m2}}<\infty $.
		\item $ \int_{\Gamma\backslash\calH}\abs{f(z)\Im(z)^{\frac m2}}^2\,dv(z)<\infty $.
	\end{enumerate}
\end{Lem}

More generally, let $ \chi:\Gamma\to\bbC^\times $ be a character of finite order. $ S_m(\Gamma,\chi) $ is defined as the space of all $ f\in S_m(\ker\chi) $ that satisfy
\begin{equation}\label{eq:051}
f\big|_m\gamma=\chi(\gamma)f,\qquad\gamma\in\Gamma. 
\end{equation}
Clearly, $ S_m(\Gamma)=S_m(\Gamma,1) $. $ S_m(\Gamma,\chi) $ is a finite-dimensional Hilbert space under the Petersson inner product
\begin{equation}\label{eq:055}
\scal{f_1}{f_2}_{\Gamma}:=\varepsilon_\Gamma^{-1}\int_{\Gamma\backslash\calH}f_1(z)\overline{f_2(z)}\Im(z)^m\,dv(z),\qquad f_1,f_2\in S_m(\Gamma,\chi). 
\end{equation}

The orthogonal projection $ S_m(\ker\chi)\twoheadrightarrow S_m(\Gamma,\chi) $ is given by
\begin{equation}\label{eq:032}
f\mapsto\sum_{\gamma\in\ker\chi\backslash\Gamma}\overline{\chi(\gamma)}f\big|_m\gamma,\qquad f\in S_m(\ker\chi). 
\end{equation}
We record another basic fact in the following lemma.
\begin{Lem}\label{lem:053}
	Let $ \sigma\in\Met $. Then, $ f\mapsto f\big|_m\sigma $ defines a unitary isomorphism $ S_m(\Gamma,\chi)\to S_m\left(\sigma^{-1}\Gamma\sigma,\chi^\sigma\right) $, where $ \chi^\sigma:=\chi\left(\sigma\spacedcdot\sigma^{-1}\right) $.
\end{Lem}

The main results of this paper concern elements of $ S_m(\Gamma,\chi) $ constructed in the form of a Poincar\' e series 
\[ P_{\Lambda\backslash\Gamma,\chi}f:=\sum_{\gamma\in\Lambda\backslash\Gamma}\overline{\chi(\gamma)}f\big|_m\gamma,
 \]where $ \Lambda $ is a subgroup of $ \Gamma $, and $ f:\calH\to\bbC $ satisfies $ f\big|_m\lambda=\chi(\lambda)f $ for all $ \lambda\in\Lambda $. We write $ P_{\Gamma,\chi} f:=P_{\{1\}\backslash\Gamma,\chi}f $ and $ P_\Gamma f:=P_{\Gamma,1}f $. 

\section{Some representation-theoretic results}

Throughout this section, let $ m\in\frac32+\bbZ_{\geq0} $.

Let $ r $ be the right regular representation of $ \Met $. For a discrete subgroup $ \Gamma $ of $ \Met $, let $ r_\Gamma $ be the unitary representation of $ \Met $ by right translations in $ L^2\left(\Gamma\backslash\Met\right) $.

\begin{Lem}\label{lem:007}
	\begin{enumerate}
		 \item There exists a unique (up to unitary equivalence) irreducible unitary representation $ \overline{\pi_m} $ of $ \Met $ that decomposes, as a representation of $ K $, into the orthogonal sum $ \bigoplus_{k\in\bbZ_{\geq0}}\chi_{m+2k} $. \label{lem:007:1}
		 \item Let $ v $ be a non-zero element of the $ \chi_m $-isotypic component of $ \overline{\pi_m} $. Then, $ \overline{\pi_m}\left(n^-\right)v=0 $, and for every $ k\in\bbZ_{\geq0} $ $ \overline{\pi_m}\left(n^+\right)^kv $ spans the $ \chi_{m+2k} $-isotypic component of $ \overline{\pi_m} $.\label{lem:007:2}
	\end{enumerate}
\end{Lem}		 
		 
\begin{proof}
	\eqref{lem:007:1} is \cite[Lemma 3-5.(1)]{ZunarNonVan2017}, and \eqref{lem:007:2} is clear from the proof of \cite[Lemma 3-5]{ZunarNonVan2017}.
\end{proof}
		 
The following lemma is central to our proof of Theorem \ref{thm:017}.

\begin{Lem}\label{lem:006}
	Let $ \Gamma $ be a discrete subgroup of $ \Met $. 
	Suppose that $ \varphi\in C^\infty\left(\Gamma\backslash\Met\right)\cap L^2\left(\Gamma\backslash\Met\right) $, $ \varphi\not\equiv0, $ has the following properties:
	\begin{enumerate}
		\item $ \varphi $ transforms on the right as $ \chi_m $. \label{lem:005:1}
		\item $ \calC \varphi=m\left(\frac m2-1\right)\varphi $. \label{lem:005:2}
	\end{enumerate}
	Then, the minimal closed subrepresentation $ H $ of $ r_\Gamma $ containing $ \varphi $ is unitarily equivalent to $ \overline{\pi_m} $, and $ \varphi $ spans its $ \chi_m $-isotypic component.
\end{Lem}

\begin{proof}
	\cite[Lemma 77]{hc} remains valid when the right regular representation of $ G $ is replaced by the representation of $ G $ by right translations in $ L^2\left(\Lambda\backslash G\right) $, where $ \Lambda $ is a discrete subgroup of $ G $. By this result, $ H $ is an orthogonal sum of finitely many closed irreducible $ \Met $-invariant subspaces. Hence, its $ (\frakg,K) $-module of $ K $-finite vectors, $ H_K $, is a direct sum of finitely many irreducible $ (\frakg,K) $-modules, and it is generated by $ \varphi $ (see \cite[Theorem 0.4]{knappvogan}). From this it follows by an elementary computation in $ H_K $, using \eqref{lem:005:1}--\eqref{lem:005:2}, that $ H_K $ is in fact an irreducible $ (\frakg,K) $-module and that it is isomorphic, as a $ K $-module, to $ \bigoplus_{k\in\bbZ_{\geq0}}\chi_{m+2k} $. Thus, $ H $ is unitarily equivalent to $ \overline{\pi_m} $ by Lemma \hyperref[lem:007:1]{\ref*{lem:007}.(\ref*{lem:007:1})}. Since $ \varphi\not\equiv0 $ belongs to its (one-dimensional) $ \chi_m $-isotypic component by \eqref{lem:005:1}, the second claim is clear.
\end{proof}
	
Next, we recall the classical lift of $ f:\calH\to\bbC $ to $ F_f:\Met\to\bbC $ defined by
\begin{equation}\label{eq:050}
F_f(\sigma):=\left(f\big|_m\sigma\right)(i),\qquad\sigma\in\Met,
\end{equation}
i.e., in Iwasawa coordinates,
\begin{equation}\label{eq:010}
F_f(n_xa_y\kappa_t)=f(x+iy)y^{\frac m2}e^{-imt},\qquad x,t\in\bbR,\ y\in\bbR_{>0}.
\end{equation}

The following result is well-known, but we could not find a convenient reference, so we provide a short proof.

\begin{Thm}\label{thm:005}
	Let $ \Gamma $ be a discrete subgroup of finite covolume in $ \Met $. Then, the lift $ f\mapsto F_f $ defines a unitary isomorphism $ S_m(\Gamma)\to\calA\left(\Gamma\backslash\Met\right)_m $, where $ \calA\left(\Gamma\backslash\Met\right)_m $ is the subspace of $ L^2\left(\Gamma\backslash\Met\right) $ consisting of all $ \varphi\in L^2\left(\Gamma\backslash\Met\right)\cap C^\infty\left(\Gamma\backslash\Met\right) $ with the following properties:
	\begin{enumerate}
		\item $ \varphi $ transforms on the right as $ \chi_m $.
		\item $ \mathcal C\varphi=m\left(\frac m2-1\right)\varphi $. 
	\end{enumerate}
	Every $ \varphi\in\calA\left(\Gamma\backslash\Met\right)_m $ is bounded.
\end{Thm}

\begin{proof}
	An elementary computation using \eqref{eq:050}, \eqref{eq:010}, \eqref{eq:046}, \eqref{eq:036}, and Lemma \ref{lem:038} shows that $ f\mapsto F_f $ is a well-defined isometry $ S_m(\Gamma)\to\calA\left(\Gamma\backslash\Met\right)_m $. To prove its surjectivity, let $ \varphi\in\calA\left(\Gamma\backslash\Met\right)_m $, $ \varphi\not\equiv0 $, and define $ f:\calH\to\bbC $, $ f(x+iy):=\varphi(n_xa_y)y^{-\frac m2} $. Obviously, $ f\in C^\infty(\calH) $ and $ F_f=\varphi $. Next, by Lemma \ref{lem:006} $ \varphi $ spans the $ \chi_m $-isotypic component of a closed subrepresentation of $ r_\Gamma $ that is unitarily equivalent to $ \overline{\pi_m} $. Thus, $ n^-\varphi=0 $ by Lemma \hyperref[lem:007:2]{\ref*{lem:007}.(\ref*{lem:007:2})}, so $ \left(\partial_x+i\partial_y\right)f=0 $ by \eqref{eq:037}, hence $ f $ is holomorphic. Furthermore, the fact that $ \varphi\in L^2\left(\Gamma\backslash\Met\right) $ implies that $ f $ satisfies \eqref{eq:003} and, by \eqref{eq:046}, that $ \int_{\Gamma\backslash\calH}\abs{f(z)\Im(z)^{\frac m2}}^2\,dv(z)<\infty $, so $ f\in S_m(\Gamma) $ by Lemma \ref{lem:038}. The same lemma implies that $ \sup_{z\in\calH}\abs{f(z)\Im(z)^{\frac m2}}<\infty $, so $ \varphi $ is bounded by \eqref{eq:010}.
\end{proof}

Next, let $ \Gamma $ be a discrete subgroup of finite covolume in $ \Met $, and let $ \chi $ be a character of $ \Gamma $ of finite order. For $ \varphi:\Met\to\bbC $, we define the Poincar\' e series
\[ \left(P_{\Gamma,\chi}\varphi\right)(\sigma):=\sum_{\gamma\in\Gamma}\overline{\chi(\gamma)}\varphi(\gamma\sigma),\qquad\sigma\in\Met. \]
We write $ P_\Gamma\varphi:=P_{\Gamma,1}\varphi $. The following lemma is elementary. Its proof is left to the reader.

\begin{Lem}\label{lem:021}
	Let $ f:\calH\to\bbC $. Then, the series $ P_{\Gamma,\chi}f $ converges absolutely (resp., absolutely and uniformly on compact sets) on $ \calH $ if and only if $ P_{\Gamma,\chi} F_f $ converges in the same way on $ \Met $, and in that case
	\begin{equation}\label{eq:022}
	F_{P_{\Gamma,\chi}f}=P_{\Gamma,\chi} F_f.
	\end{equation}
\end{Lem}	

\section{Proof of Theorem \ref{thm:017}}

In the following lemma we recall some results of \cite{ZunarNonVan2017}.
		 
\begin{Lem}\label{lem:045}
	Let $ m\in\frac32+\bbZ_{\geq0} $ and $ k\in\bbZ_{\geq0} $.
	\begin{enumerate} 
		 \item We define $ f_{k,m}:\calH\to\bbC $,
		 \[ f_{k,m}(z):=(2i)^m\frac{(z-i)^k}{(z+i)^{m+k}}. \]
		 $ F_{k,m}:=F_{f_{k,m}} $ is a (unique up to a multiplicative constant) matrix coefficient of $ \overline{\pi_m} $ that transforms on the right as $ \chi_m $ and on the left as $ \chi_{m+2k} $.\label{lem:045:1} 
		 \item $ \calC F_{k,m}=m\left(\frac m2-1\right)F_{k,m} $.\label{lem:045:2}
		 \item We have
		 \[ F_{k,m}\left(\kappa_{\theta_1}h_t\kappa_{\theta_2}\right)=\chi_{m+2k}\left(\kappa_{\theta_1}\right)\frac{\tanh^k(t)}{\cosh^m(t)}\chi_m\left(\kappa_{\theta_2}\right),\quad\theta_1,\theta_2\in\bbR,\ t\in\bbR_{\geq0}. \]\label{lem:045:3}
		 \item If $ m\in\frac52+\bbZ_{\geq0} $, then $ F_{k,m}\in L^1\left(\Met\right) $.\label{lem:045:4}
		 \item Suppose $ m\in\frac52+\bbZ_{\geq0} $. Let $ \Gamma $ be a discrete subgroup of finite covolume in $ \Met $. Then, the series $ \sum_{\gamma\in\Gamma}\abs{F_{k,m}(\gamma\spacedcdot)} $ converges uniformly on compact sets in $ \Met $, and $ P_\Gamma F_{k,m}\in\calA\left(\Gamma\backslash\Met\right)_m $.\label{lem:045:5}
	\end{enumerate}
	\end{Lem}

\begin{proof}
	\eqref{lem:045:1} is \cite[Proposition 4-7]{ZunarNonVan2017}, \eqref{lem:045:2} follows from \cite[Lemma 4-4.(3)]{ZunarNonVan2017}, \eqref{lem:045:3} is \cite[Lemma 4-9]{ZunarNonVan2017}, \eqref{lem:045:4} is \cite[Lemma 4-10]{ZunarNonVan2017}, and \eqref{lem:045:5} is clear from the proof of \cite[Lemma 6-2]{ZunarNonVan2017}.
\end{proof}

Next, we prepare a few technical results for the proof of Theorem \ref{thm:017}.

\begin{Lem}\label{lem:008}
	Let $ m\in\frac32+\bbZ_{\geq0} $ and $ k\in\bbZ_{\geq0} $. Then, we have the following:
	\begin{enumerate}
		\item $\norm{F_{k,m}}_{L^2\left(\Met\right)}^2=\frac{4\pi k!}{\prod_{r=0}^{k}(m-1+r)}.$\label{lem:008:1}
		\item Let $ f\in\Hol(\calH) $. Then, for all $ x,y,t\in\bbR $ with $ y>0 $,
		\begin{equation}\label{eq:009}
		\left(\left(n^+\right)^kF_f\right)(n_xa_y\kappa_t)=\chi_{m+2k}(\kappa_t)y^{\frac m2}\sum_{l=0}^k\binom kl(2iy)^l\left(\prod_{r=l+1}^{k}(m-1+r)\right)f^{(l)}(x+iy).
		\end{equation}\label{lem:008:2}
		\item $ \left(\left(n^+\right)^kF_{k,m}\right)(1)=k! $.\label{lem:008:3}
	\end{enumerate}
\end{Lem}

\begin{proof}
	\eqref{lem:008:1} By Lemma \hyperref[lem:045:3]{\ref*{lem:045}.(\ref*{lem:045:3})} and \eqref{eq:008}, we have
	\[ \norm{F_{k,m}}_{L^2\left(\Met\right)}^2=\frac1{4\pi}\int_0^{4\pi}\int_0^\infty\int_0^{4\pi}\frac{\tanh^{2k}(t)}{\cosh^{2m}(t)}\sinh(2t)\,d\theta_1\,dt\,d\theta_2, \]
	which, substituting $ x=\tanh^2(t) $ and using the identities $ \sinh(2t)=2\sinh(t)\cosh(t) $ and $ \frac1{\cosh^2(t)}=1-\tanh^2(t) $, equals
	\[ 4\pi\int_0^1x^k(1-x)^{m-2}\,dx=\frac{4\pi k!}{\prod_{r=0}^{k}(m-1+r)}. \]
	The last equality is obtained by $ k $-fold partial integration.

	\eqref{lem:008:2} This is proved by induction on $ k\in\bbZ_{\geq0} $ using \eqref{eq:011} and noting that in the case when $ k=0 $ \eqref{eq:009} is \eqref{eq:010}.
	
	\eqref{lem:008:3} Since $ F_{k,m}=F_{f_{k,m}} $, \eqref{lem:008:3} is just \eqref{lem:008:2} applied to $ f=f_{k,m} $ with $ x=t=0 $ and $ y=1 $.   
\end{proof}

Let $ \Gamma $ be a discrete subgroup of $ \Met $. For $ F\in L^1\left(\Met\right) $ and $ \varphi\in L^2\left(\Gamma\backslash\Met\right) $, $ r_\Gamma(F)\varphi\in L^2\left(\Gamma\backslash\Met\right) $ is standardly defined by the following condition:
\[ \scal{r_\Gamma(F)\varphi}{\phi}_{L^2\left(\Gamma\backslash\Met\right)}=\int_\Met F(y)\scal{r_\Gamma(y)\varphi}{\phi}_{L^2\left(\Gamma\backslash\Met\right)}\,d\mu_\Met(y) \]
for all $ \phi\in L^2\left(\Gamma\backslash\Met\right) $. It is well-known that
\begin{equation}\label{eq:012}
\left(r_\Gamma(F)\varphi\right)(x)=\int_\Met F(y)\varphi(xy)\,d\mu_\Met(y)
\end{equation}
for almost all $ x\in\Met $. The following lemma is immediate.

\begin{Lem}\label{lem:016}
	Let $ \Gamma $ be a discrete subgroup of $ \Met $. Let $ F\in L^1\left(\Met\right) $ and $ \varphi\in L^2\left(\Gamma\backslash\Met\right) $. If $ \varphi $ is continuous and bounded, then the integral in \eqref{eq:012} converges for every $ x\in\Met $, and \eqref{eq:012} expresses the continuous representative of $ r_\Gamma(F)\varphi $. 
\end{Lem}

Now we are ready to prove the main result of this paper:

\begin{Thm}\label{thm:017}
	Let $ \Gamma $ be a discrete subgroup of finite covolume in $ \Met $. Let $ m\in\frac52+\bbZ_{\geq0} $, $ k\in\bbZ_{\geq0} $, and $ \varphi\in\calA\left(\Gamma\backslash\Met\right)_m $. Then,
	\begin{equation}\label{eq:023}
	\scal\varphi{P_\Gamma F_{k,m}}_{L^2\left(\Gamma\backslash\Met\right)}=\frac{4\pi}{\prod_{r=0}^{k}(m-1+r)}\left(\left(n^+\right)^k\varphi\right)(1). 
	\end{equation}
\end{Thm}

\begin{proof}
	The case when $ \varphi\equiv0 $ is trivial, so suppose that $ \varphi\not\equiv0 $. We have
	\begin{align}
		\scal\varphi{P_\Gamma F_{k,m}}_{L^2\left(\Gamma\backslash\Met\right)}&=\int_{\Gamma\backslash\Met}\varphi(\sigma)\overline{\sum_{\gamma\in\Gamma}F_{k,m}(\gamma\sigma)}\,d\mu_{\Gamma\backslash\Met}(\sigma) \nonumber\\
		&=\int_{\Gamma\backslash\Met}\sum_{\gamma\in\Gamma}\varphi(\gamma\sigma)\overline{F_{k,m}(\gamma\sigma)}\,d\mu_{\Gamma\backslash\Met}(\sigma) \label{eq:013}\\
		&=\int_\Met\varphi\overline{F_{k,m}}\,d\mu_\Met.\nonumber
	\end{align}
	
	Now, $ F_{k,m}\in L^1\left(\Met\right) $ by Lemma \hyperref[lem:045:4]{\ref*{lem:045}.(\ref*{lem:045:4})}, and $ \varphi $ is continuous and bounded by Theorem \ref{thm:005}. Thus, by Lemma \ref{lem:016}, $ r_\Gamma\left(\overline{F_{k,m}}\right)\varphi\in L^2\left(\Gamma\backslash\Met\right)\cap C\left(\Gamma\backslash\Met\right) $ is given by
	\begin{equation}\label{eq:014}
	\left(r_\Gamma\left(\overline{F_{k,m}}\right)\varphi\right)(x)=\int_\Met\overline{F_{k,m}(y)}\varphi(xy)\,d\mu_\Met(y),\qquad x\in\Met.
	\end{equation}
	In particular,
	\[ \left(r_\Gamma\left(\overline{F_{k,m}}\right)\varphi\right)(1)=\int_\Met\varphi\overline{F_{k,m}}\,d\mu_\Met, \]
	so \eqref{eq:013} implies that
	\begin{equation}\label{eq:018}
	\scal\varphi{P_\Gamma F_{k,m}}_{L^2\left(\Gamma\backslash\Met\right)}=\left(r_\Gamma\left(\overline{F_{k,m}}\right)\varphi\right)(1). 
	\end{equation}
	
	To compute $ \left(r_\Gamma\left(\overline{F_{k,m}}\right)\varphi\right)(1) $, we note that by Lemma \ref{lem:006} $ \varphi $ generates the $ \chi_m $-isotypic component of a closed subrepresentation $ H_\varphi $ of $ r_\Gamma $ that is unitarily equivalent to $ \overline{\pi_m} $. Clearly, $ r_\Gamma\left(\overline{F_{k,m}}\right)\varphi\in H_\varphi $. In fact, $ r_\Gamma\left(\overline{F_{k,m}}\right)\varphi $ belongs to the $ \chi_{m+2k} $-isotypic component of $ H_\varphi $: since $ F_{k,m} $ transforms on the left as $ \chi_{m+2k} $ by Lemma \hyperref[lem:045:1]{\ref*{lem:045}.(\ref*{lem:045:1})}, we have 
	\begin{align*}
		\left(r_\Gamma\left(\overline{F_{k,m}}\right)\varphi\right)(x\kappa)&\overset{\eqref{eq:014}}=\int_\Met\overline{F_{k,m}(y)}\varphi(x\kappa y)\,d\mu_\Met(y)\\
		&\ =\ \int_\Met\overline{F_{k,m}\left(\kappa^{-1}y\right)}\varphi(xy)\,d\mu_\Met(y)\\
		&\ =\ \chi_{m+2k}(\kappa)\int_\Met\overline{F_{k,m}(y)}\varphi(xy)\,d\mu_\Met(y)\\
		&\overset{\eqref{eq:014}}=\chi_{m+2k}(\kappa)\left(r_\Gamma\left(\overline{F_{k,m}}\right)\varphi\right)(x)
	\end{align*}
	for all $ x\in\Met $ and $ \kappa\in K $. Hence, by Lemma \hyperref[lem:007:2]{\ref*{lem:007}.(\ref*{lem:007:2})}, 
	\begin{equation}\label{eq:015}
	r_\Gamma\left(\overline{F_{k,m}}\right)\varphi=\lambda\left(n^+\right)^k\varphi\qquad\text{for some }\lambda\in\bbC.
	\end{equation}
	
	To calculate $ \lambda $, we apply Lemma \ref{lem:006}, with $ \Gamma=\{1\} $, to $ F_{k,m} $. ($ F_{k,m} $ satisfies all conditions of Lemma \ref{lem:006} by Lemmas \hyperref[lem:045:1]{\ref*{lem:045}.(\ref*{lem:045:1})}--\eqref{lem:045:2} and \hyperref[lem:008:1]{\ref*{lem:008}.(\ref*{lem:008:1})}.) We obtain that $ F_{k,m} $ spans the $ \chi_m $-isotypic component of a closed subrepresentation $ H_{F_{k,m}} $ of $ r $ that is unitarily equivalent to $ \overline{\pi_m} $. Let $ \Phi:H_{\varphi}\to H_{F_{k,m}} $ be a unitary equivalence. Since $ \varphi $ and $ F_{k,m} $ span the $ \chi_m $-isotypic components of, respectively, $ H_\varphi $ and $ H_{F_{k,m}} $, we have
	\begin{equation}\label{eq:016}
		\Phi\varphi=s F_{k,m}\quad\text{for some }s\in\bbC^\times.
	\end{equation} 
	By applying $ \Phi $ to both sides of \eqref{eq:015}, we obtain
	\[ r\left(\overline{F_{k,m}}\right)\Phi\varphi=\lambda\left(n^+\right)^k\Phi\varphi, \]
	hence by \eqref{eq:016}
	\[ r\left(\overline{F_{k,m}}\right)F_{k,m}=\lambda\left(n^+\right)^kF_{k,m}. \]
	By evaluating (the continuous representatives of) both sides at $ 1\in\Met $ and using that $ \left(r\left(\overline{F_{k,m}}\right)F_{k,m}\right)(1)=\norm{F_{k,m}}_{L^2\left(\Met\right)}^2 $ by Lemma \ref{lem:016}, we obtain
	\begin{equation}\label{eq:020}
	\lambda=\frac{\norm{F_{k,m}}_{L^2\left(\Met\right)}^2}{\left(\left(n^+\right)^kF_{k,m}\right)(1)}=\frac{4\pi}{\prod_{r=0}^{k}(m-1+r)}
	\end{equation}
	by Lemma \hyperref[lem:008:1]{\ref*{lem:008}.(\ref*{lem:008:1})} and \eqref{lem:008:3}.
	
	Thus,
	\begin{align*}
		\scal\varphi{P_\Gamma F_{k,m}}_{L^2\left(\Gamma\backslash\Met\right)}&\overset{\eqref{eq:018}}=\left(r_\Gamma\left(\overline{F_{k,m}}\right)\varphi\right)(1)\\
		&\overset{\eqref{eq:015}}=\lambda\left(\left(n^+\right)^k\varphi\right)(1)\\
		&\overset{\eqref{eq:020}}=\frac{4\pi}{\prod_{r=0}^{k}(m-1+r)}\left(\left(n^+\right)^k\varphi\right)(1).\qedhere
	\end{align*}
\end{proof}

\section{Cusp forms \eqref{eq:073} and their non-vanishing}

Throughout this section, let $ m\in\frac52+\bbZ_{\geq0} $.

\begin{Thm}\label{thm:020}
	Let $ \Gamma $ be a discrete subgroup of finite covolume in $ \Met $, $ \chi:\Gamma\to\bbC^\times $ a character of finite order, and $ k\in\bbZ_{\geq0} $. Then:
	\begin{enumerate}
		\item The series $ P_{\Gamma,\chi}f_{k,m} $ converges absolutely and uniformly on compact sets in $ \calH $.\label{thm:020:1} 
		\item $ P_{\Gamma,\chi}f_{k,m}\in S_m(\Gamma,\chi) $. \label{thm:020:2}
		\item For every $ f\in S_m(\Gamma,\chi) $,
		\[ \scal f{P_{\Gamma,\chi}f_{k,m}}_\Gamma=\sum_{l=0}^k\binom kl(2i)^l\frac{4\pi}{\prod_{r=0}^{l}(m-1+r)}f^{(l)}(i). \]\label{thm:020:3}	
		\item $ \left\{P_{\Gamma,\chi} f_{n,m}:n\in\bbZ_{\geq0}\right\} $ spans $ S_m(\Gamma,\chi) $.\label{thm:020:4}
	\end{enumerate}
\end{Thm}

\begin{proof}
	\eqref{thm:020:1} By Lemma \ref{lem:021} it suffices to prove that the series $ P_{\Gamma,\chi} F_{f_{k,m}}=P_{\Gamma,\chi} F_{k,m} $ converges absolutely and uniformly on compact sets in $ \Met $, which is clear from Lemma \hyperref[lem:045:5]{\ref*{lem:045}.(\ref*{lem:045:5})}.
	
	Next, we prove \eqref{thm:020:2}--\eqref{thm:020:4} in the case when $ \chi=1 $:
	
	\eqref{thm:020:2} Since $ F_{P_\Gamma f_{k,m}}=P_\Gamma F_{k,m} $ belongs to $ \calA\left(\Gamma\backslash\Met\right)_m $ by  Lemma \hyperref[lem:045:5]{\ref*{lem:045}.(\ref*{lem:045:5})}, by Theorem \ref{thm:005} $ P_{\Gamma}f_{k,m} $ belongs to $ S_m(\Gamma) $.
	
	\eqref{thm:020:3} Let $ f\in S_m(\Gamma) $. We have, by Theorem \ref{thm:005},
	\begin{align*}
		\scal f{P_\Gamma f_{k,m}}_\Gamma&=\scal{F_f}{F_{P_\Gamma f_{k,m}}}_{L^2\left(\Gamma\backslash\Met\right)}\\
		&\overset{\eqref{eq:022}}=\scal{F_f}{P_\Gamma F_{k,m}}_{L^2\left(\Gamma\backslash\Met\right)}\\
		&\overset{\eqref{eq:023}}=\frac{4\pi}{\prod_{r=0}^{k}(m-1+r)}\left(\left(n^+\right)^kF_f\right)(1)\\
		&\overset{\eqref{eq:009}}=\sum_{l=0}^k\binom kl(2i)^l\frac{4\pi}{\prod_{r=0}^{l}(m-1+r)}f^{(l)}(i).
	\end{align*}
	
	\eqref{thm:020:4} It suffices to show that every $ f\in S_m(\Gamma) $ satisfying $ \scal f{P_\Gamma f_{n,m}}_\Gamma=0 $ for all $ n\in\bbZ_{\geq0} $ is identically zero. Indeed, from \eqref{thm:020:3} it follows by induction on $ n\in\bbZ_{\geq0} $ that such an $ f $ satisfies $ f^{(n)}(i)=0 $ for all $ n\in\bbZ_{\geq0} $, so $ f $ is identically zero since $ f\in\Hol(\calH) $.
	
	Now, since the orthogonal projection $ S_m(\ker\chi)\twoheadrightarrow S_m(\Gamma,\chi) $ given by \eqref{eq:032} maps $ P_{\ker\chi} f_{k,m} $ to $ P_{\Gamma,\chi}f_{k,m} $, the claims \eqref{thm:020:2}--\eqref{thm:020:4} in the case when $ \chi\neq1 $ follow from the proven ones about $ P_{\ker\chi}f_{k,m} $.
\end{proof}

Next, we give a result on the non-vanishing of cusp forms $ P_{\Gamma,\chi} f_{k,m} $ in the case when $ P(\Gamma)\subseteq\SL2(\bbZ) $. Let us denote by $ \M(a,b) $ the median of the beta distribution with parameters $ a,b\in\bbR_{>0} $, i.e., the unique $ \M(a,b)\in\left]0,1\right[ $ such that
\[ \int_0^{\M(a,b)}x^{a-1}(1-x)^{b-1}\,dx=\int_{\M(a,b)}^1x^{a-1}(1-x)^{b-1}\,dx. \]

\begin{Thm}\label{thm:030}
	Let $ N\in\bbZ_{>0} $ and $ k\in\bbZ_{\geq0} $. Let $ \Gamma $ be a subgroup of finite index in $ P^{-1}(\Gamma(N)) $, and let $ \chi:\Gamma\to\bbC^\times $ be a character of finite order. Suppose that
	\begin{equation}\label{eq:052}
	 \chi\big|_{\Gamma\cap Z\left(\Met\right)}=\chi_m\big|_{\Gamma\cap Z\left(\Met\right)}  
	\end{equation}
	(otherwise $ S_m(\Gamma,\chi)=0 $ by \eqref{eq:051}). If
	\begin{equation}\label{eq:031}
	N>\frac{4\,\M\left(\frac k2+1,\frac m2-1\right)^{\frac12}}{1-\M\left(\frac k2+1,\frac m2-1\right)}, 
	\end{equation}
	then $ P_{\Gamma,\chi}f_{k,m} $ is not identically zero.  	
\end{Thm}

\begin{proof}
	It suffices to prove the non-vanishing of $ F_{P_{\Gamma,\chi}f_{k,m}}=P_{\Gamma,\chi}F_{k,m} $. We do this by applying to $ P_{\Gamma,\chi}F_{k,m} $ the non-vanishing criterion \cite[Lemma 2-1]{MuicIJNT} with $ \Gamma_1=\{1\} $ and $ \Gamma_2=\Gamma\cap Z\left(\Met\right) $: $ F_{k,m} $ satisfies the condition (1) of \cite[Lemma 2-1]{MuicIJNT} since it transforms on the right as $ \chi_m $  and \eqref{eq:052} holds. A compact set $ C $ satisfying the conditions (2)--(3) of \cite[Lemma 2-1]{MuicIJNT} can be found using \eqref{eq:031} exactly as in the proof of \cite[Proposition 6-7]{ZunarNonVan2017}.
\end{proof}

To illustrate the strength of Theorem \ref{thm:030}, we can use some well-known properties of $ \M(a,b) $ \cite[Lemma 6-12]{ZunarNonVan2017} to obtain the following variant of \cite[Corollary 6-18]{ZunarNonVan2017}.

\begin{Cor}\label{cor:070}
	Let $ N\in\bbZ_{>0} $ and $ k\in\bbZ_{\geq0} $. Let $ \Gamma $ be a subgroup of finite index in $ P^{-1}(\Gamma(N)) $. Let $ \chi:\Gamma\to\bbC^\times $ be a character of finite order such that \eqref{eq:052} holds. Then, $ P_{\Gamma,\chi}f_{k,m} $ is not identically zero if one of the following holds:
	\begin{enumerate}
		\item $ k=0 $ and $ N>4\cdot2^{\frac1{m-2}}\sqrt{4^{\frac1{m-2}}-1} $
		\item $ m=4 $ and $ {\displaystyle N>\frac4{2^{\frac1{k+2}}-2^{-\frac1{k+2}}} } $
		\item $ 0<k\leq m-4 $ and $ N\geq4\sqrt{\frac{k+2}{m-2}\left(1+\frac{k+2}{m-2}\right)} $
		\item $ 0<m-4\leq k $ and $ N\geq4\sqrt{\frac{k}{m-4}\left(1+\frac{k}{m-4}\right)} $.
	\end{enumerate}
\end{Cor}

\section{Cusp forms \eqref{eq:074}}\label{sec:047}

Throughout this section, let $ \Gamma $ be a discrete subgroup of finite covolume in $ \Met $, $ \chi:\Gamma\to\bbC^\times $ a character of finite order, and $ m\in\frac52+\bbZ_{\geq0} $.

For every $ k\in\bbZ_{\geq0} $ and $ \xi\in\calH $, we define $ \delta_{k,m,\xi}:\calH\to\bbC $,
\begin{equation}\label{eq:057}
\delta_{k,m,\xi}(z):=\frac{(2i)^m}{4\pi}\left(\prod_{r=0}^{k}(m-1+r)\right)\frac{1}{\left(z-\overline\xi\right)^{m+k}}.
\end{equation}
Note that $ \delta_{k,m,\xi}(z)=\left(\frac{d}{d\overline\xi}\right)^k\delta_{0,m,\xi}(z) $.

\begin{Prop}\label{prop:026}
	Let $ k\in\bbZ_{\geq0} $ and $ \xi\in\calH $. Then, the Poincar\' e series
	\begin{equation}\label{eq:054}
	\Delta_{\Gamma,k,m,\xi,\chi}(z):=\left(P_{\Gamma,\chi}\delta_{k,m,\xi}\right)(z)=\frac{(2i)^m}{4\pi}\left(\prod_{r=0}^{k}(m-1+r)\right)\sum_{\gamma\in\Gamma}\frac{\overline{\chi(\gamma)}}{\left(\gamma.z-\overline\xi\right)^{m+k}}\eta_\gamma(z)^{-2m} 
	\end{equation}
	converges absolutely and uniformly on compact sets in $ \calH $ and belongs to $ S_m(\Gamma,\chi) $. 
\end{Prop}

\begin{proof}
	This can be proved by applying the obvious half-integral weight variant of \cite[Theorems 2.6.6.(1) and 2.6.7]{miyake}. We give an alternative proof.
	Note that
	\begin{equation}\label{eq:024}
	f_{k,m}=\sum_{l=0}^k\binom kl(-2i)^l\frac{4\pi}{\prod_{r=0}^{l}(m-1+r)}\delta_{l,m,i},\qquad k\in\bbZ_{\geq0},
	\end{equation}
	hence by the binomial inversion formula
	\[ \delta_{k,m,i}=\frac{\prod_{r=0}^{k}(m-1+r)}{4\pi(2i)^k}\sum_{l=0}^k\binom kl(-1)^lf_{l,m},\qquad k\in\bbZ_{\geq0}, \]
	so the claim in the case when $ \xi=i $ follows from Theorem \hyperref[thm:020:1]{\ref*{thm:020}.(\ref*{thm:020:1})} and \eqref{thm:020:2}. Now the claim for general $ \xi=x+iy\in\calH $ (with $ x,y\in\bbR $) is clear, using Lemma \ref{lem:053}, from the identity
	\begin{equation}\label{eq:049}
	\Delta_{\Gamma,k,m,\xi,\chi}=y^{-\frac m2-k}\Delta_{\left(n_xa_y\right)^{-1}\Gamma n_xa_y,k,m,i,\chi^{n_xa_y}}\big|_m\left(n_xa_y\right)^{-1},
	\end{equation}
	which is easily checked by following definitions.
\end{proof}

The following technical lemmas will be used in our analytic proof of Theorem \ref{thm:048}.

\begin{Lem}\label{lem:064}
	Let $ (X,dx) $ be a measure space. Let $ D $ be a domain in $ \bbC $. Suppose that $ f:D\times X\to\bbC $ is a measurable function with the following properties:
	\begin{enumerate}
		\item For every $ x\in X $, $ f(\spacedcdot,x) $ is holomorphic on $ D $.
		\item\label{lem:064:2} For every circle $ C\subseteq D $, $ \int_{C\times X}\abs{f(z,x)}\,d(z,x)<\infty $.
	\end{enumerate} 
	Then, $ F:D\to\bbC $, 
	\begin{equation}\label{eq:063}
	F(z):=\int_Xf(z,x)\,dx,
	\end{equation}
	is well-defined and holomorphic on $ D $, and we have
	\begin{equation}\label{eq:062}
	F^{(k)}(z)=\int_X\left(\frac d{d\zeta}\right)^kf(\zeta,x)\Big|_{\zeta=z}\,dx,\qquad z\in D,\ k\in\bbZ_{>0}.
	\end{equation}
\end{Lem}

\begin{proof}
	Without \eqref{eq:062}, this is \cite[Lemma 6.1.5]{miyake}. To prove \eqref{eq:062}, let $ z\in D $ and fix $ \delta\in\bbR_{>0} $ such that $ \left\{\zeta\in\bbC:\abs{\zeta-z}\leq\delta\right\}\subseteq D $. Let $ k\in\bbZ_{>0} $. We have
	\begin{align*}
		F^{(k)}(z)&\ =\ \frac{k!}{2\pi i}\int_{\abs{\zeta-z}=\delta}\frac{F(\zeta)}{(\zeta-z)^{k+1}}\,d\zeta\\
		&\overset{\eqref{eq:063}}=\int_X\left(\frac{k!}{2\pi i}\int_{\abs{\zeta-z}=\delta}\frac{f(\zeta,x)}{(\zeta-z)^{k+1}}\,d\zeta\right)\,dx\\
		&\ =\ \int_X\left(\frac d{d\zeta}\right)^kf(\zeta,x)\Big|_{\zeta=z}\,dx
	\end{align*}
	by applying the Cauchy integral formula for derivatives in the first and the last, and Fubini's theorem in the second equality.
\end{proof}

\begin{Lem}\label{lem:084}
	Let $ f\in S_m(\Gamma,\chi) $. Then, the function $ I_f:\calH\to\bbC $,
	\[ I_f(\xi):=\scal f{\Delta_{\Gamma,0,m,\xi,\chi}}_\Gamma, \]
	is holomorphic, and $ I_f^{(k)}(\xi)=\scal f{\Delta_{\Gamma,k,m,\xi,\chi}}_\Gamma $ for all $ \xi\in\calH $ and $ k\in\bbZ_{>0}$.
\end{Lem}

\begin{proof}
	For every $ k\in\bbZ_{\geq0} $ and $ \xi\in\calH $, 
	\begin{align}
	\scal f{\Delta_{\Gamma,k,m,\xi,\chi}}_\Gamma
	&\underset{\eqref{eq:054}}{\overset{\eqref{eq:055}}=}\varepsilon_\Gamma^{-1}\int_{\Gamma\backslash\calH}f(z)\sum_{\gamma\in\Gamma}\chi(\gamma)\overline{\left(\delta_{k,m,\xi}\big|_m\gamma\right)(z)}\Im(z)^m\,dv(z)\nonumber\\
	&\overset{\eqref{eq:051}}=\varepsilon_\Gamma^{-1}\int_{\Gamma\backslash\calH}\sum_{\gamma\in\Gamma}\left(f\big|_m\gamma\right)(z)\overline{\left(\delta_{k,m,\xi}\big|_m\gamma\right)(z)}\Im(z)^m\,dv(z)\nonumber\\
	&\overset{\eqref{eq:056}}=\varepsilon_\Gamma^{-1}\int_{\Gamma\backslash\calH}\sum_{\gamma\in\Gamma}f(\gamma.z)\overline{\delta_{k,m,\xi}(\gamma.z)}\Im(\gamma.z)^m\,dv(z)\nonumber\\
	&\ =\ \int_\calH f(z)\overline{\delta_{k,m,\xi}(z)}\Im(z)^m\,dv(z)\nonumber\\
	&\overset{\eqref{eq:057}}=\frac{(-2i)^m}{4\pi}\left(\prod_{r=0}^k(m-1+r)\right)\int_\calH\frac{f(z)}{\left(\overline z-\xi\right)^{m+k}}\Im(z)^m\,dv(z).\label{eq:060}
	\end{align}
	The claim of the lemma follows from \eqref{eq:060} by Lemma \ref{lem:064}. The condition \eqref{lem:064:2} of Lemma \ref{lem:064} is satisfied since
	\begin{align}
	\int_\calH\frac{\abs{f(z)}}{\abs{\overline z-\xi}^{m+k}}\Im(z)^m\,dv(z)&\leq\left(\sup_{z\in\calH}\abs{f(z)\Im(z)^{\frac m2}}\right)\int_\calH\frac{\Im(z)^{\frac m2}}{\abs{z-\overline\xi}^{m+k}}\,dv(z)\label{eq:080}\\
	&=\left(\sup_{z\in\calH}\abs{f(z)\Im(z)^{\frac m2}}\right)\int_\calH\frac{\Im(z)^{\frac m2}}{\abs{z+i}^{m+k}}\,dv(z) \frac1{\Im(\xi)^{\frac m2+k}}\nonumber 
	\end{align}
	(applying the substitution $ z\mapsto n_{\Re(\xi)}a_{\Im(\xi)}.z $ for the last equality), and the right-hand side is obviously bounded for $ \xi $ in any circle $ C\subseteq\calH $.
\end{proof}

\begin{Thm}\label{thm:048}
	We have
	\begin{equation}\label{eq:029}
	\scal f{\Delta_{\Gamma,k,m,\xi,\chi}}_\Gamma=f^{(k)}(\xi),\qquad f\in S_m(\Gamma,\chi),\ k\in\bbZ_{\geq0},\ \xi\in\calH.
	\end{equation}
	For every $ \xi\in\calH $, $ \left\{\Delta_{\Gamma,k,m,\xi,\chi}:k\in\bbZ_{\geq0}\right\} $ spans $ S_m(\Gamma,\chi) $.
\end{Thm}

\begin{proof}
	Using \eqref{eq:024}, Theorem \hyperref[thm:020:3]{\ref*{thm:020}.(\ref*{thm:020:3})} can be written in the following way: for all $ f\in S_m(\Gamma,\chi) $ and $ k\in\bbZ_{\geq0} $,
	\[ \sum_{l=0}^k\binom kl(2i)^l\frac{4\pi}{\prod_{r=0}^l(m-1+r)}\scal f{\Delta_{\Gamma,l,m,i,\chi}}_\Gamma=\sum_{l=0}^k\binom kl(2i)^l\frac{4\pi}{\prod_{r=0}^l(m-1+r)}f^{(l)}(i). \]
	This implies, by induction on $ k\in\bbZ_{\geq0} $, that, for all $ f\in S_m(\Gamma,\chi) $ and $ k\in\bbZ_{\geq0} $, 
	\begin{equation}\label{eq:028}
	\scal f{\Delta_{\Gamma,k,m,i,\chi}}_\Gamma=f^{(k)}(i). 
	\end{equation}
	
	From here, one can obtain \eqref{eq:029} for general $ \xi=x+iy\in\calH $ (with $ x,y\in\bbR $) in two ways. The first is algebraic (cf.\ the proof of \cite[Lemma 3-8]{MuicInner}): Let $ f\in S_m(\Gamma,\chi) $. By taking the $ k $th derivative at $ z=i $ of the both sides of the equality $ f(x+yz)=y^{-\frac m2}\left(f\big|_mn_xa_y\right)(z) $, we obtain, using Lemma \ref{lem:053},
	\begin{align*}
		f^{(k)}(\xi)&\ =\ y^{-\frac m2-k}\left(f\big|_mn_xa_y\right)^{(k)}(i)\\
		&\overset{\eqref{eq:028}}=y^{-\frac m2-k}\scal{f\big|_mn_xa_y}{\Delta_{\left(n_xa_y\right)^{-1}\Gamma n_xa_y,k,m,i,\chi^{n_xa_y}}}_{\left(n_xa_y\right)^{-1}\Gamma n_xa_y}\\
		&\ =\ \scal f{y^{-\frac m2-k}\Delta_{\left(n_xa_y\right)^{-1}\Gamma n_xa_y,k,m,i,\chi^{n_xa_y}}\big|_m\left(n_xa_y\right)^{-1}}_\Gamma\\
		&\overset{\eqref{eq:049}}=\scal f{\Delta_{\Gamma,k,m,\xi,\chi}}_\Gamma.
	\end{align*} 
	
	A second way to obtain \eqref{eq:029} from \eqref{eq:028} is analytic: Let $ f\in S_m(\Gamma,\chi) $. By Lemma \ref{lem:084}, \eqref{eq:028} shows that 
	\[ I_f^{(k)}(i)=f^{(k)}(i),\qquad k\in\bbZ_{\geq0}, \]
	i.e., $ f $ and $ I_{f} $ have the same Taylor expansion at $ i $. Since both are holomorphic on $ \calH $, it follows by the uniqueness of analytic continuation that
	\[ I_f^{(k)}(\xi)=f^{(k)}(\xi),\qquad \xi\in\calH,\ k\in\bbZ_{\geq0}, \]
	and this is \eqref{eq:029} by Lemma \ref{lem:084}.    	 
	
	The second claim of the theorem follows from \eqref{eq:029} as in the proof of Theorem \hyperref[thm:020:4]{\ref*{thm:020}.(\ref*{thm:020:4})}.
\end{proof}

\eqref{eq:029} and \eqref{eq:060} prove the following integral formula:

\begin{Cor}\label{cor:071}
	Let $ f\in S_m(\Gamma,\chi) $. Then, for all $ k\in\bbZ_{\geq0} $ and $ \xi\in\calH $,
	\[ f^{(k)}(\xi)=\frac{(-2i)^m}{4\pi}\left(\prod_{r=0}^k(m-1+r)\right)\int_\calH\frac{f(z)}{\left(\overline z-\xi\right)^{m+k}}\Im(z)^m\,dv(z). \]
\end{Cor}

More generally, Corollary \ref{cor:071} holds for every $ f\in\Hol(\calH) $ such that $ \sup_{z\in\calH}\abs{f(z)\Im(z)^{\frac m2}}<\infty $. This follows from the half-integral weight version of \cite[Theorem 6.2.2]{miyake}. As a simple application of Corollary \ref{cor:071}, we prove:

\begin{Cor}\label{cor:081}
	Let $ f\in S_m(\Gamma,\chi) $. Then, for every $ k\in\bbZ_{\geq0} $,
	\[ \sup_{\xi\in\calH}\abs{f^{(k)}(\xi)\Im(\xi)^{\frac m2+k}}<\infty. \]
\end{Cor}

\begin{proof}
	By Corollary \ref{cor:071} and \eqref{eq:080},
	\[ \sup_{\xi\in\calH}\abs{f^{(k)}(\xi)\Im(\xi)^{\frac m2+k}}\leq\frac{2^m}{4\pi}\left(\prod_{r=0}^k(m-1+r)\right)\left(\sup_{z\in\calH}\abs{f(z)\Im(z)^{\frac m2}}\right)\int_\calH\frac{\Im(z)^{\frac m2}}{\abs{z+i}^{m+k}}\,dv(z), \]
	and the right-hand side is finite by Lemma \ref{lem:038}. 
\end{proof}

Now we can easily prove the following result (cf.\ \cite[(1-5)]{MuicLFunk}):

\begin{Prop}\label{prop:082}
	Let $ k\in\bbZ_{\geq0} $. Then,
	\[ \sup_{z,\xi\in\calH}\Im(\xi)^{\frac m2+k}\Im(z)^{\frac m2}\abs{\Delta_{\Gamma,k,m,\xi,\chi}(z)}<\infty.  \]
\end{Prop}

\begin{proof}
	Let us fix an orthonormal basis $ \left\{f_1,\ldots,f_d\right\} $ of $ S_m(\Gamma,\chi) $. We have
	\[ \Delta_{\Gamma,k,m,\xi,\chi}(z)=\sum_{l=1}^d\scal{\Delta_{\Gamma,k,m,\xi,\chi}}{f_l}_\Gamma f_l(z)\overset{\eqref{eq:029}}=\sum_{l=1}^d\overline{f_l^{(k)}(\xi)}f_l(z), \]
	hence
	\[ \sup_{z,\xi\in\calH}\Im(\xi)^{\frac m2+k}\Im(z)^{\frac m2}\abs{\Delta_{\Gamma,k,m,\xi,\chi}(z)}\leq
	\sum_{l=1}^d\left(\sup_{\xi\in\calH}\abs{f_l^{(k)}(\xi)\Im(\xi)^{\frac m2+k}}\right)\left(\sup_{z\in\calH}\abs{f_l(z)\Im(z)^{\frac m2}}\right), \]
	and the right-hand side is finite by Corollary \ref{cor:081}.
\end{proof}

\section{Two expansions of cusp forms \eqref{eq:074}}\label{sec:048}

Throughout this section, let $ \Gamma $ be a discrete subgroup of finite covolume in $ \Met $, $ \chi:\Gamma\to\bbC^\times $ a character of finite order, and $ m\in\frac52+\bbZ_{\geq0} $. Moreover, suppose that $ \infty $ is a cusp of $ P(\Gamma) $ and that 
\[ \eta_\gamma(z)^{-2m}=\chi(\gamma),\qquad\gamma\in\Gamma_\infty,\ z\in\calH. \]
Let $ h\in\bbR_{>0} $ such that 
\[ Z\left(\Met\right)\Gamma_\infty=Z\left(\Met\right)\left<n_h\right>. \]

By the half-integral weight version of \cite[Theorem 2.6.9]{miyake}, for every $ n\in\bbZ_{>0} $ the classical Poincar\' e series
\[ \psi_{\Gamma,n,m,\chi}:=P_{\Gamma_\infty\backslash\Gamma,\chi}e^{2\pi in\frac\spacedcdot h} \]
converges absolutely and uniformly on compact sets in $ \calH $, and $ \psi_{\Gamma,n,m,\chi}\in S_m(\Gamma,\chi) $. Moreover, by the half-integral weight version of \cite[Theorem 2.6.10]{miyake}, every $ f\in S_m(\Gamma,\chi) $ has the following Fourier expansion:
\begin{equation}\label{eq:087}
f(z)=\frac{\varepsilon_\Gamma(4\pi)^{m-1}}{\Gamma(m-1)h^m}\sum_{n=1}^\infty n^{m-1}\scal f{\psi_{\Gamma,n,m,\chi}}_\Gamma e^{2\pi in\frac zh},\qquad z\in\calH.
\end{equation}
Here we use the standard notation for the gamma function: $ \Gamma(x):=\int_0^\infty t^{x-1}e^{-t}\,dt $, $ x\in\bbR_{>0} $.

Theorem \ref{thm:035} provides the Fourier expansion of cusp forms $ \Delta_{\Gamma,k,m,\xi,\chi} $ and their expansion in a series of classical Poincar\' e series. It is a half-integral weight variant of \cite[Theorem 3-5]{MuicCurves}. Lemma \ref{lem:034} resolves the convergence issues of its proof.

We define a norm $ \norm{\spacedcdot}_{\Gamma,1} $ on $ S_m(\Gamma,\chi) $ by
\[ \norm{f}_{\Gamma,1}:=\int_{\Gamma\backslash\calH}\abs{f(z)\Im(z)^{\frac m2}}\,dv(z),\qquad f\in S_m(\Gamma,\chi). \] 

\begin{Lem}\label{lem:034}
	Let $ k\in\bbZ_{\geq0} $ and $ \xi\in\calH $. Then, the series
	\begin{equation}\label{eq:065}
	\sum_{n=1}^\infty n^{m+k-1}e^{-2\pi in\frac{\overline\xi}{h}}\psi_{\Gamma,n,m,\chi} 
	\end{equation}
	converges:
	\begin{enumerate}
		\item absolutely in the norm $ \norm{\spacedcdot}_{\Gamma,1} $\label{lem:033:1}
		\item absolutely and uniformly on compact sets in $ \calH $\label{lem:033:2}
		\item in the topology of $ S_m(\Gamma,\chi) $.\label{lem:033:3}
	\end{enumerate}
\end{Lem}

\begin{proof}
	\eqref{lem:033:1} implies the absolute convergence of \eqref{eq:065} at every $ z\in\calH $ by \cite[Corollary 2.6.2]{miyake}. \eqref{lem:033:1} also implies the rest of the claims \eqref{lem:033:2} and \eqref{lem:033:3} since $ S_m(\Gamma,\chi) $ is finite-dimensional. 
	
	To prove \eqref{lem:033:1}, observe that
	\begin{align*}
		\norm{\psi_{\Gamma,n,m,\chi}}_{\Gamma,1}&\leq\int_{\Gamma\backslash\calH}\sum_{\gamma\in\Gamma_\infty\backslash\Gamma}\abs{\left(e^{2\pi in\frac{\spacedcdot}h}\big|_m\gamma\right)(z)\Im( z)^{\frac m2}}\,dv(z)\\
		&\overset{\eqref{eq:056}}=\int_{\Gamma\backslash\calH}\sum_{\gamma\in\Gamma_\infty\backslash\Gamma}\abs{e^{2\pi in\frac{\gamma.z}h}\Im\left(\gamma.z\right)^{\frac m2}}\,dv(z)
		=\int_{\Gamma_\infty\backslash\calH}\abs{e^{2\pi in\frac zh}\Im(z)^{\frac m2}}\,dv(z)\\
		&=\int_0^h\int_0^\infty e^{-2\pi n\frac yh}y^{\frac m2-2}\,dy\,dx
		=h\left(\frac h{2\pi n}\right)^{\frac m2-1}\Gamma\left(\frac m2-1\right),
	\end{align*}
	so
	\[ \sum_{n=1}^\infty \norm{n^{m+k-1}e^{-2\pi in\frac{\overline\xi}{h}}\psi_{\Gamma,n,m,\chi}}_{\Gamma,1}\leq \frac{h^{\frac m2}}{(2\pi)^{\frac m2-1}}\Gamma\left(\frac m2-1\right)\sum_{n=1}^\infty n^{\frac m2+k}e^{-2\pi n\frac{\Im(\xi)}h}, \]
	and the right-hand side is finite by d'Alembert's ratio test.
\end{proof}

\begin{Thm}\label{thm:035}
	Let $ k\in\bbZ_{\geq0} $ and $ \xi\in\calH $. Then:
	\begin{enumerate}
		\item $ \Delta_{\Gamma,k,m,\xi,\chi} $ has the following Fourier expansion:\label{thm:035:1}
		\begin{equation}\label{eq:042}
		\Delta_{\Gamma,k,m,\xi,\chi}(z)=\frac{\varepsilon_\Gamma\left(4\pi\right)^{m-1}}{\Gamma(m-1)h^{m}}\sum_{n=1}^\infty n^{m-1}\overline{\psi_{\Gamma,n,m,\chi}^{(k)}(\xi)}e^{2\pi in\frac{z}{h}},\qquad z\in\calH. 
		\end{equation}
		\item We have
		\begin{equation}\label{eq:040}
		\Delta_{\Gamma,k,m,\xi,\chi}(z)=\frac{\varepsilon_\Gamma\left(4\pi\right)^{m-1}(-2\pi i)^k}{\Gamma(m-1)h^{m+k}}\sum_{n=1}^\infty n^{m+k-1}e^{-2\pi in\frac{\overline\xi}{h}}\psi_{\Gamma,n,m,\chi}(z),\qquad z\in\calH.
		\end{equation}
		The right-hand side converges in $ S_m(\Gamma,\chi) $ and absolutely and uniformly on compact sets in $ \calH $.\label{thm:035:2}
	\end{enumerate}
\end{Thm}

\begin{proof}
	This can be proved analogously to the proof of \cite[Theorem 3-5]{MuicCurves}, all convergence issues being settled by Lemma \ref{lem:034}. We provide a shorter proof:
	
	\eqref{thm:035:1} \eqref{eq:042} follows from \eqref{eq:087} since $ \scal {\Delta_{\Gamma,k,m,\xi,\chi}}{\psi_{\Gamma,n,m,\chi}}_\Gamma\overset{\eqref{eq:029}}=\overline{\psi_{\Gamma,n,m,\chi}^{(k)}(\xi)} $.
	
	\eqref{thm:035:2} We have
	\begin{align*}
		\Delta_{\Gamma,k,m,\xi,\chi}(z)&\overset{\eqref{eq:029}}=\scal{\Delta_{\Gamma,k,m,\xi,\chi}}{\Delta_{\Gamma,0,m,z,\chi}}_\Gamma\overset{\eqref{eq:029}}=\overline{\Delta_{\Gamma,0,m,z,\chi}^{(k)}(\xi)}\\
		&\overset{\eqref{eq:042}}=\frac{\varepsilon_\Gamma\left(4\pi\right)^{m-1}}{\Gamma(m-1)h^{m}}\sum_{n=1}^\infty n^{m-1}\psi_{\Gamma,n,m,\chi}(z)\overline{\left(\frac{2\pi in}h\right)^ke^{2\pi in\frac{\xi}{h}}}\\
		&=\frac{\varepsilon_\Gamma\left(4\pi\right)^{m-1}(-2\pi i)^k}{\Gamma(m-1)h^{m+k}}\sum_{n=1}^\infty n^{m+k-1}e^{-2\pi in\frac{\overline\xi}{h}}\psi_{\Gamma,n,m,\chi}(z),\qquad z\in\calH.
	\end{align*}
	The convergence claim follows from Lemma \ref{lem:034}.
\end{proof}

Now we can easily prove some bounds on the derivatives of classical Poincar\' e series (cf.\ \cite[Theorem 1-2]{MuicLFunk}):

\begin{Cor}\label{cor:090}
	Let $ k\in\bbZ_{\geq0} $. Then,
	\[ \sup_{\substack{\xi\in\calH,\\n\in\bbZ_{>0}}}n^{\frac m2-1}\Im(\xi)^{\frac m2+k}\abs{\psi_{\Gamma,n,m,\chi}^{(k)}(\xi)}<\infty. \]
\end{Cor}

\begin{proof}
	Let us fix an orthonormal basis $ \left\{f_1,f_2,\ldots,f_d\right\} $ of $ S_m(\Gamma,\chi) $, and for each $ l\in\left\{1,2,\ldots,d\right\} $ let $ f_l(z) = \sum_{n=1}^\infty a_n(f_l)e^{2\pi in\frac zh} $ be the Fourier expansion of $ f_l $. We have
	\[ \Delta_{\Gamma,k,m,\xi,\chi}(z)=\sum_{l=1}^d\scal{\Delta_{\Gamma,k,m,\xi,\chi}}{f_l}_\Gamma f_l(z)
	= \sum_{n=1}^\infty\left(\sum_{l=1}^d\overline{f_l^{(k)}(\xi)}a_n(f_l)\right)e^{2\pi in\frac zh},\quad z,\xi\in\calH. \]
	hence by Theorem \hyperref[thm:035:1]{\ref*{thm:035}.(\ref*{thm:035:1})}
	\[ \frac{\varepsilon_\Gamma\left(4\pi\right)^{m-1}}{\Gamma(m-1)h^{m}}n^{m-1}\overline{\psi_{\Gamma,n,m,\chi}^{(k)}(\xi)}=\sum_{l=1}^d\overline{f_l^{(k)}(\xi)}a_n(f_l),\qquad n\in\bbZ_{>0},\ \xi\in\calH. \]
	Thus,
	\[ \sup_{\substack{\xi\in\calH,\\n\in\bbZ_{>0}}}n^{\frac m2-1}\Im(\xi)^{\frac m2+k}\abs{\psi_{\Gamma,n,m,\chi}^{(k)}(\xi)}\leq\frac{\Gamma(m-1)h^m}{\varepsilon_\Gamma(4\pi)^{m-1}}\sum_{l=1}^d\left(\sup_{\xi\in\calH}\abs{f_l^{(k)}(\xi)\Im(\xi)^{\frac m2+k}}\right)\left(\sup_{n\in\bbZ_{>0}}\frac{\abs{a_n(f_l)}}{n^{\frac m2}}\right), \]
	and the right-hand side is finite by Corollary \ref{cor:081} and by the half-integral weight version of \cite[Corollary 2.1.6]{miyake}.
\end{proof}

\section{Application to cusp forms for $ \Gamma_0(N) $}\label{sec:084}

We define the automorphic factor $ J:\Gamma_0(4)\times\ScptH\to\bbC $,
\[ J(\gamma,z):=\frac{\Theta(\gamma.z)}{\Theta(z)}, \]
where $ \Theta\in\Hol(\ScptH) $ is given by $ \Theta(z):=\sum_{n\in\bbZ}e^{2\pi in^2z} $. An explicit formula for $ J $ is given by \cite[III.(4.2)]{koblitz}. It easily implies that for every $ N\in4\bbZ_{>0} $
\[ \bfGamma_0(N):=\left\{\left(\gamma,J(\gamma,\spacedcdot)\right):\gamma\in\Gamma_0(N)\cap\Gamma_1(4)\right\} \]
is a discrete subgroup of finite covolume in $ \Met $. 

Let $ m\in\frac52+\bbZ_{\geq0} $ and $ N\in4\bbZ_{>0} $. Let $ \chi $ be an even Dirichlet character modulo $ N $. We identify $ \chi $ with the character of $ \Gamma_0(N) $ given by $ \begin{pmatrix}a&b\\c&d\end{pmatrix}\mapsto \chi(d) $ for all $ \begin{pmatrix}a&b\\c&d\end{pmatrix}\in\Gamma_0(N) $, and with the character of $ \bfGamma_0(N) $ given by $ \left(\gamma,J\left(\gamma,\spacedcdot\right)\right)\mapsto\chi(\gamma) $ for all $ \gamma\in\Gamma_0(N)\cap\Gamma_1(4). $
Finally, we define
\[ S_m(N,\chi):=S_m(\bfGamma_0(N),\chi). \]
This definition of $ S_m(N,\chi) $ is equivalent to the one given in \cite{shimura}. (In \cite{shimura}, $ S_m(N,\chi) $ is defined regardless of the parity of $ \chi $, but turns out to be $ 0 $ if $ \chi $ is odd.) The Petersson inner product on $ S_m(N,\chi) $ is
\[ \scal fg_{\bfGamma_0(N)}=\int_{\Gamma_0(N)\backslash\calH}f(z)\overline{g(z)}\Im( z)^m\,dv(z),\qquad f,g\in S_m(N,\chi), \]
and we have, for all $ k\in\bbZ_{\geq0} $ and $ \xi,z\in\calH $,
\[ \Delta_{\bfGamma_0(N),k,m,\xi,\chi}(z)=\frac{(2i)^m}{8\pi}\left(\prod_{r=0}^{k}(m-1+r)\right)\sum_{\gamma\in\Gamma_0(N)}\frac{\overline{\chi(\gamma)}}{\left(\gamma.z-\overline\xi\right)^{m+k}}J(\gamma,z)^{-2m}. \]

$ \bfGamma_0(N) $ and $ \chi $ satisfy the assumptions of the first paragraph of Section \ref{sec:048}, hence we have the classical Poincar\' e series
\[ \psi_{\bfGamma_0(N),n,m,\chi}(z)=\sum_{\gamma\in\Gamma_0(N)_\infty\backslash\Gamma_0(N)}\overline{\chi(\gamma)}e^{2\pi in\gamma.z}J(\gamma,z)^{-2m},\qquad z\in\calH,\ n\in\bbZ_{>0}, \]
and the cusp forms $ \Delta_{\bfGamma_0(N),k,m,\xi,\chi} $ have the expansion \eqref{eq:040} in a series of classical Poincar\' e series. As a final application of our results, in Corollary \ref{cor:066} we express the action of Hecke operators of half-integral weight on $ \Delta_{\bfGamma_0(N),k,m,\xi,\chi} $ in terms of \eqref{eq:040} (cf.\ \cite[Lemma 5-8]{MuicCurves}). 

For every prime number $ p $, the Hecke operator $ T_{p^2,m,\chi}:S_m(N,\chi)\to S_m(N,\chi) $ is given by 
\[ \sum_{n=1}^\infty a(n)e^{2\pi inz}\Big|T_{p^2,m,\chi}:=\sum_{n=1}^\infty b(n)e^{2\pi inz}, \]
where
\begin{equation}\label{eq:043}
b(n):=a\left(p^2n\right)+\left(\frac{-1}p\right)^{m-\frac12}\chi(p)\left(\frac np\right)p^{m-\frac32}a(n)+\chi\left(p^2\right)p^{2m-2}a\left(n/p^2\right)
\end{equation}
\cite[Theorem 1.7]{shimura}. Here we understand that $ a\left(n/p^2\right)=0 $ if $ p^2\nmid n $, while $ \left(\frac\spacedcdot p\right) $ is the usual Legendre symbol if $ p $ is odd and is identically zero if $ p=2 $. 

If $ p\nmid N $, then
\begin{equation}\label{eq:041}
\scal {f\big|T_{p^2,m,\chi}}g_{\bfGamma_0(N)}=\chi\left(p^2\right)\scal f{g\big|T_{p^2,m,\chi}}_{\bfGamma_0(N)},\qquad f,g\in S_m\left(N,\chi\right).
\end{equation}
This enables us to prove:

\begin{Cor}\label{cor:066}
	Let $ N\in4\bbZ_{>0} $, $ m\in\frac52+\bbZ_{\geq0} $, $ k\in\bbZ_{\geq0} $, and $ \xi\in\calH $. Let $ \chi $ be an even Dirichlet character modulo $ N $. Then, for every prime number $ p $ such that $ p\nmid N $, $ T_{p^2,m,\chi} $ maps the cusp form
	\begin{equation}\label{eq:067}
	\Delta_{\bfGamma_0(N),k,m,\xi,\chi}(z)=\frac{\left(4\pi\right)^{m-1}(-2\pi i)^k}{\Gamma(m-1)}\sum_{n=1}^\infty n^{m+k-1}e^{-2\pi in\overline\xi}\psi_{\bfGamma_0(N),n,m,\chi}(z)
	\end{equation}
	to 
	\begin{equation}\label{eq:044}
	\left(\Delta_{\bfGamma_0(N),k,m,\xi,\chi}\big|T_{p^2,m,\chi}\right)(z)=\frac{\left(4\pi\right)^{m-1}(-2\pi i)^k}{\Gamma(m-1)}\sum_{n=1}^\infty n^{m+k-1}E_{p,k,n,m,\chi}(\xi)\psi_{\bfGamma_0(N),n,m,\chi}(z), 
	\end{equation}
	where
	\[ E_{p,k,n,m,\chi}(\xi):=\mathbbm1_{p^2\bbZ}(n)\frac{\chi\left(p^2\right)}{p^{2k}}e^{-2\pi i\frac n{p^2}\overline\xi}+\left(\frac{-1}p\right)^{m-\frac12}\chi(p)\left(\frac np\right)p^{m-\frac32}e^{-2\pi in\overline\xi}+p^{2m+2k-2}e^{-2\pi ip^2n\overline\xi}. \]
	Here $ \mathbbm1_{p^2\bbZ} $ is the characteristic function of $ p^2\bbZ\subseteq\bbZ $, and $ \left(\frac\spacedcdot p\right) $ is the usual Legendre symbol.
\end{Cor}

\begin{proof}
	\eqref{eq:067} is a special case of \eqref{eq:040}.
	The proof of \eqref{eq:044} is analogous to that of \cite[Lemma 5-8]{MuicCurves}. For every $ z\in\calH $, we have
	\begin{align*}
		\left(\Delta_{\bfGamma_0(N),k,m,\xi,\chi}\big|T_{p^2,m,\chi}\right)(z)&\overset{\eqref{eq:029}}=\scal{\Delta_{\bfGamma_0(N),k,m,\xi,\chi}\big|T_{p^2,m,\chi}}{\Delta_{\bfGamma_0(N),0,m,z,\chi}}_{\bfGamma_0(N)}\\
		&\overset{\eqref{eq:041}}=\chi\left(p^2\right)\scal{\Delta_{\bfGamma_0(N),k,m,\xi,\chi}}{\Delta_{\bfGamma_0(N),0,m,z,\chi}\big|T_{p^2,m,\chi}}_{\bfGamma_0(N)}\\
		&\overset{\eqref{eq:029}}=\chi\left(p^2\right)\overline{\left(\Delta_{\bfGamma_0(N),0,m,z,\chi}\big|T_{p^2,m,\chi}\right)^{(k)}(\xi)}.
	\end{align*}
	By \eqref{eq:042} and \eqref{eq:043}, the right-hand side equals
	\begin{align*}
		\frac{(4\pi)^{m-1}}{\Gamma(m-1)}\sum_{n=1}^\infty(-2\pi in)^k
		\Big(&\chi\left(p^2\right)\left(p^2n\right)^{m-1}\psi_{\bfGamma_0(N),p^2n,m,\chi}(z)\\
		&+\left(\frac{-1}p\right)^{m-\frac12}\chi(p)\left(\frac np\right) p^{m-\frac32}n^{m-1}\psi_{\bfGamma_0(N),n,m,\chi}(z)\\
		&+p^{2m-2}\left(n/p^2\right)^{m-1}\psi_{\bfGamma_0(N),n/p^2,m,\chi}(z)\Big)e^{-2\pi in\overline\xi}.
	\end{align*} 
	By rearranging this sum to be over the index $ n $ in $ \psi_{\bfGamma_0(N),n,m,\chi}(z) $, we obtain \eqref{eq:044}. The rearrangement is valid by Lemma \ref{lem:034}.
\end{proof}

\end{document}